\documentclass{amsart}
\usepackage[utf8]{inputenc}
\usepackage{amsmath,pdfsync,verbatim,graphicx,epstopdf,enumerate,latexsym,amssymb,mathtools,thmtools,thm-restate,accents,cancel,amsthm}
\usepackage[usenames]{color}
\usepackage[abbrev,nobysame]{amsrefs}
\mathtoolsset{showonlyrefs} 
\usepackage[colorlinks=true]{hyperref}
\hypersetup{allcolors=blue}
\usepackage[framemethod=tikz]{mdframed}
\usepackage{tikz}
\usetikzlibrary{calc,trees,positioning,arrows,chains,shapes.geometric,decorations.pathreplacing,decorations.pathmorphing,shapes,matrix,shapes.symbols}
\newtheorem{theorem}{Theorem}[section]
\newtheorem{corollary}[theorem]{Corollary}
\newtheorem{lemma}[theorem]{Lemma}

\newtheorem{remark}[theorem]{Remark}
\newcommand{\Rn}{\mathbb{R}^n}
\newcommand{\R}{\mathbb{R}}
\newcommand{\D}{\mathrm{d}}
\newcommand{\Lc}{\mathcal{L}}
\newcommand{\Rb}{\mathbb{R}}

\newcommand{\A}{\alpha}

\newcommand{\Cb}{\mathbb{C}}

\newcommand{\ve}{\varepsilon}

\newcounter{assump}
\newenvironment{assump}{%
  \refstepcounter{assump}%
  \paragraph{\textbf{Assumption~\theassump}}%
  \enumerate
}{%
  \endenumerate
}

\renewcommand{\d}{\delta}
\newcommand{\wt}{\widetilde}
\newcommand{\Ac}{\mathcal{A}}

\newcommand{\Ec}{\mathcal{E}}
\newcommand{\Fc}{\mathcal{F}}

\newcommand{\Vc}{\mathcal{V}}
\newcommand{\Wc}{\mathcal{W}}

\newcommand{\Nb}{\mathbb{N}}

\newcommand{\I}{\mathrm{i}}

\renewcommand{\O}{\Omega}
\newcommand{\PD}{\partial}

\usepackage[normalem]{ulem}
\newcommand{\C}{\mathbb{C}}

\newcommand{\p}{\partial}
\newcommand{\lb}{\left(}
\newcommand{\rb}{\right)}

\title[Density results for biharmonic functions]{Density results of biharmonic functions on symmetric tensor fields and their applications to inverse problems}

\author{Divyansh Agrawal}
\address{Centre for Applicable Mathematics, Tata Institute of Fundamental Research, India.}
\email{agrawald@tifrbng.res.in, agrdiv01@gmail.com \newline
\href{https://sites.google.com/view/agrawald/home}{https://sites.google.com/view/agrawald/home}}

\author{Sombuddha Bhattacharyya}
\address{Department of Mathematics, Indian Institute of Science Education and Research, Bhopal, India.
}
\email{sombuddha@iiserb.ac.in, \newline
\href{https://sites.google.com/iiserb.ac.in/sombuddha/home}{https://sites.google.com/iiserb.ac.in/sombuddha/home}}

\author{Pranav Kumar}
\address{Department of Mathematics, Indian Institute of Science Education and Research, Bhopal, India.}
\email{pranav19@iiserb.ac.in, pk958675maths@gmail.com}

\begin{document}
\begin{abstract}
In this article we discuss density of products of biharmonic functions vanishing on an arbitrarily small part of the boundary. We prove that one can use three or more such biharmonic functions to construct a dense subset of smooth symmetric tensor fields up to order three, in a bounded domain. Furthermore, as an application of the density results, in dimension two or higher, we solve a partial data inverse problem for a biharmonic operator with nonlinear anisotropic third and lower order perturbations. For the inverse problem, we take the Dirichlet data to be supported in an arbitrarily small open set of the boundary and measure the Neumann data on the same set. Note that the analogous problem for linear perturbations are still unknown. 
So far, partial data problems recovering nonlinear perturbations were studied only up to vector fields \cite{lai2020partial}.
The full data analogues of the inverse problem has recently been studied in \cite{BKSU2023} for three or higher dimensions.
\end{abstract}


\maketitle

\section{Introduction and the main results}
Let us consider $\Omega \subset \mathbb{R}^n$, $n\geq 2$ to be an open bounded domain with smooth boundary $\partial\Omega$. Let $\Gamma \subset \partial\Omega$ be a proper closed set, i.e. $\Gamma^c := \partial\Omega\setminus \Gamma$ is a nonempty open set.
In this article we prove that using three or more smooth biharmonic functions in $\Omega$, vanishing on $\Gamma$, one can construct dense subsets of symmetric smooth tensor fields up to order $3$ on $\Omega$. 
In the context of inverse problems, the set $\Gamma\subset\PD\O$ can be regarded as an inaccessible part of the boundary.
Let us start with discussing a particular case of the density result that we prove.
Consider a symmetric $3$-tensor field $A^{(3)}(x)$ whose components $A^{(3)}_{i_1 i_2 i_3}(x)$ are smooth functions on $\overline{\Omega}$, where $i_1,i_2,i_3=1,\cdots,n$.
Let $v_k \in C^{\infty}(\overline{\Omega})$ denotes biharmonic functions, that is $(-\Delta)^2v_k=0$ in $\Omega$, for $k=0,1,2$. Let us further assume that $v_k|_{\Gamma} = \partial_{\nu} v_k|_{\Gamma} = 0$, for $k=0,1,2$.
We denote $D := \frac{1}{\I} \lb \frac{\PD}{\PD x_1}, \dots, \frac{\PD}{\PD x_n}\rb, D_j = \frac{1}{\I} \frac{\PD}{\PD x_j}$, $D^{(l)}_{i_1 \dots i_l} = \frac{1}{\I^l} \frac{\PD^l}{\PD x_{i_1} \dots \PD x_{i_l}}$ and $\partial_{\nu} = \nu\cdot \nabla_x$ where $\nu$ is the unit outward normal vector to the boundary.
If possible, let us assume
\begin{equation}\label{Intro_1}
\int_{\Omega} \sum_{i_1,i_2,i_3 = 1}^{n} A^{(3)}_{i_1 i_2 i_3}(x) \left(v_2 D^{(3)}_{i_1i_2i_3} v_1 + v_1 D^{(3)}_{i_1i_2i_3} v_2 \right)v_0 \, dx = 0,
\end{equation}
for all such $v_k$, $k=0,1,2$; then we prove that $A^{(3)}$ vanishes identically in $\Omega$.
In order to state the density result in full generality let us define the following notations.
Let us define the set
\begin{equation}\label{eq_Ec}
\Ec_{\Gamma} \coloneqq \{ u \in C^\infty(\overline{\O}): (-\Delta)^2u=0\text{ in }\O \text{ and } (u, \PD_\nu u)\lvert_{\Gamma} = 0 \}.
\end{equation}
For $l=1,2,3$, denote $A^{(l)}$ to be symmetric $l$-tensor fields with components in $C^{\infty}(\overline{\Omega})$, while $A^{(0)}$ denotes a function. Let us denote
\begin{equation}\label{Op_Int}
\begin{aligned}
\Ac_m (v_0, \dots, v_m) 
:= \sum\limits_{l=0}^3 \sum\limits_{i_1 \dots i_l = 1}^n A^{(l)}_{i_1 \dots i_l}(x) \lb \sum_{k=1}^m D^{(l)}_{i_1 \dots i_l} v_k \left( \prod\limits_{\substack{j=1\\ j \neq k}}^{m} v_j \right)\rb v_0,
\end{aligned}
\end{equation}
where $v_k \in \Ec_{\Gamma}$ for $k=0,1,\cdots,m$ and $m\geq 2$. Here and henceforth, for notational conveniences, a tensor field of order $0$ is used to denote a function, and the above summation for $l=0$ should be treated as a sum over an empty set, i.e., no summation is performed.
With these notations, we state the first density result for $\Ac_m$.
\begin{theorem}\label{MR1}
    Let $\Ac_m, \Ec_\Gamma$ be as above and $m \geq 2$. The integral identity
    \[
    \int\limits_{\O} \Ac_{m}(v_0, \dots, v_m) \, \D x = 0, \qquad \text{for all} ~v_i \in \Ec_{\Gamma}; \quad i=0,\cdots,m,
    \]
    implies that $\Ac_{m} \equiv 0$, i.e., the coefficients $A^{(l)}$ vanish identically in $\O$ for $l=0,\cdots,3$.
\end{theorem}

The density results in this article stem from the linearized Calder\'on problem, which asks that if for a function $q \in L^{\infty}(\Omega)$, one has
\[
\int_{\Omega} q(x) u(x) v(x) \,dx = 0, \quad \forall u,v \in \{ w \in C^{\infty}(\overline{\Omega}) : (-\Delta)w = 0 \mbox{ in }\Omega \mbox{ and }w|_{\Gamma}=0\},
\]
then can one conclude that $q\equiv 0$ in $\Omega$?
A positive answer to the above question is proved in \cite{DSFKSU} where the authors have shown that the product of two harmonic functions in $\Omega$, vanishing on $\Gamma$, is dense in $L^2(\O)$ topology.
A related result worth mentioning is \cite{Katya-Uhlmann-gradient-nonlinearities} where the authors proved density in $L^2(\Omega)$ of inner product of the gradients of two harmonic functions vanishing on $\Gamma$. Some relevant density results in the context of inverse problems for semilinear elliptic equations have been discussed in \cites{Katya-Uhlmann-semilinear-elliptic,lai2020partial,Ali_Catalin_1,Ali_Catalin_2,Ali_Cataline_3}.
Recently, \cite{agrawal2023linearized} proved that using two biharmonic functions with boundary restrictions, one can prove density result concerning a smooth isotropic 2 tensor, a vector field, and a function. More precisely, if for a vector field $A^{(1)} \in C^{\infty}(\overline{\Omega};\mathbb{R}^n)$ and functions $A^{(2)},A^{(0)} \in C^{\infty}(\overline{\Omega})$ one has
\begin{equation}\label{agarwal_23}
\int_{\Omega} \left(A^{(2)}(\Delta u) + A^{(1)}\cdot\nabla u + A^{(0)}u\right)v\, dx = 0, \quad \mbox{ for all } u,v,\in \Ec_{\Gamma},
\end{equation}
then $A^{(l)}=0$ in $\Omega$ for $l=0,1,2$.
Observe that, one cannot conclude anything about the coefficient $A^{(2)}$ using only harmonic functions and that justifies considering biharmonic functions for third and second order tensors.
One possible generalization of the existing density results for a symmetric three tensor is to consider the following operator
\[
\Ac^{\sharp}(x,D) u \coloneqq \sum_{l=0}^3 \sum_{i_1 \dots i_l = 1}^n A^{(l)}_{i_1 \dots i_l}(x) D^{(l)}_{i_1 \dots i_l} u.
\]
We also adopt the convention that $D^{(0)}u = u$.
We prove the following density result for the operator $\Ac^{\sharp}(x,D)$.
\begin{theorem}\label{MR2}
    Let $\Ac^{\sharp}, \Ec_{\Gamma}$ be as above and $m\geq 2$. The integral identity
    \[
    \int_{\O} \left(\Ac^{\sharp}(x,D) v_0\right)~ v_1 \dots v_m \, \D x = 0, \qquad \text{for all } ~v_i \in \Ec_{\Gamma}, 
    \]
    implies that $A^{(l)}(x)=0$ in $\O$ for $l=0,\cdots,3$.
\end{theorem}

In contrast to previous results, considering three or more biharmonic function with boundary restrictions, we prove density results for general symmetric tensors fields up to order three.
Note that this result is not true for $m=1$ (see \cite{SahooSalo}).
Another reason for considering the operator $\Ac^{\sharp}$ is that it allows us to get rid of the intertwining of the functions $v_1, \dots, v_m$, as compared to operator $\Ac_{m}$. 
We show that Theorem \ref{MR1} can be obtained from Theorem \ref{MR2}; see Lemma~\ref{equiv}.

Note that one cannot hope to generalize the above Theorem \ref{MR1} and \ref{MR2} for a fourth order symmetric tensor using biharmonic functions. For instance, consider a symmetric fourth order tensor field
\[
A^{(4)}_{ijkl} := \begin{cases}
    f(x) \quad \mbox{if } i=j=k=l,\\
    0 \quad \mbox{otherwise }
\end{cases}
\quad \mbox{where } f(x) \in C^{\infty}(\overline{\O}),
\]
for which the corresponding integral equation in Theorem \ref{MR2} reads
\[
\sum_{i,j,k,l=1}^n \int_{\Omega} A^{(4)}_{ijkl} \left(D^{(4)}_{ijkl} v_0\right)v_1\cdots v_m \, dx 
= \int_{\Omega} f(x) \left((-\Delta)^2v_0 \right) v_1\cdots v_m \, dx = 0.
\]
It is straightforward to see that biharmonic functions do not help in concluding anything about $A^{(4)}$.
This justifies the restriction on the order of the tensor fields in our analysis. Though density results for symmetric tensors of order four or higher remain open, the authors believe that using polyharmonic functions, i.e., $(-\Delta)^l u=0$ for $l\geq 3$, our techniques may help to generalize this result for tensor fields of order four and higher.

\subsection{An application to Inverse Problems}\label{Sec_MR_IP}
As an application of the density theorems, we solve an inverse problem of recovering all the lower order nonlinear symmetric tensorial perturbations of the bi-Laplacian operator $(-\Delta)^2$, from the knowledge of its Dirichlet to Neumann map only on a part of the boundary. 
Let $\Omega\subset \mathbb{R}^n$, $n\geq 2$ be an open bounded domain with smooth boundary $\partial\Omega$.
We consider a semilinear biharmonic partial differential operator $\mathcal{L}$ as
\begin{equation}\label{Non_Linear_Operator}
\begin{aligned}
\mathcal{L}_{A^{(0)},A^{(1)},A^{(2)},A^{(3)}}u 
:= (-\Delta)^{2}u+\sum_{l=0}^{3}\sum_{i_{1},\cdots, i_{l}=1}^{n} A^{(l)}_{i_{1}\cdots i_{l}}(x,u)D^{(l)}_{i_{1}\cdots i_{l}}u,
\end{aligned}
\end{equation}
where $A^{(l)}(x,z)$ are holomorphic in the $z$ variable and smooth in the $x$ variable. More precisely, the coefficients $A^{(l)}: \overline{\Omega}\times \Cb\to S^l$, for $l=0,\cdots,3$ satisfy:
\begin{assump}\label{Assumption}
\item\label{Asmp_1} The map $\C\ni z\mapsto A^{(l)}(\cdot,z)$ is holomorphic with values in $C^{\infty}(\overline{\Omega};S^l)$,
\item\label{Asmp_2} $A^{(l)}(x,0)=0$, for $l=0,\cdots,3$,
for all $x\in \overline{\O}$. 
\end{assump}
Here, $S^l$ denotes the space of all symmetric $l$-tensors in $\Rn$ for $l=0,\cdots,3$. See \cite{BKSU2023} for an inverse problem concerning such an operator in dimensions $3$ or higher.
Consider the Dirichlet problem
\begin{equation}\label{Dirichlet_Problem}
\begin{aligned}
\begin{cases}
\mathcal{L}_{A^{(0)},A^{(1)},A^{(2)}, A^{(3)}} u =0 & \text { in } \Omega,\\
\left(u, \partial_\nu u\right)\qquad \quad \,\,\, = \left(f, g\right) &\text { on } \partial \Omega.
\end{cases}
\end{aligned}
\end{equation}
Note that under the Assumption~\ref{Assumption}, \eqref{Dirichlet_Problem} is well-posed in dimensions $n \geq 2$ (see Subsection~\ref{well-posed}) for $(f,g) \in V_{\delta}(\PD\O)$ for $\delta>0$ small enough, where 
\begin{equation}\label{eq_v_delta}
V_{\delta}(\p\O) := \{(f,g)\in C^{4,\A}(\p\O) \times C^{3,\A}(\p\O) : \lVert f \rVert_{C^{4,\A}(\p\O)} + \lVert g \rVert_{C^{3,\A}(\p\O)} <\delta \}.
\end{equation}

Let $\Gamma\subset \partial\Omega$ be an arbitrary non-empty, closed, proper subset. We define the partial Dirichlet--to--Neumann map(DN-map) associated to \eqref{Dirichlet_Problem} as
\begin{equation}\label{Partial-DN-map}
\mathcal{N}^{\Gamma}_{A^{(0)},A^{(1)},A^{(2)},A^{(3)}}(f, g) :=(\left.\partial_{\nu}^{2}u, \partial_{\nu}^{3} u\right)|_{\Gamma^c},
\end{equation}
where $f,g \in V_{\delta}(\p\O)$, $f|_{\Gamma},g|_{\Gamma} = 0$ and $u$ is the unique solution to \eqref{Dirichlet_Problem}.
We prove that the partial DN-map $\mathcal{N}^{\Gamma}_{A^{(0)},A^{(1)},A^{(2)}A^{(3)}}$ uniquely determines the nonlinear perturbations $A^{(0)},A^{(1)}$, $A^{(2)}$ and $A^{(3)}$ in $\Omega$.
Here we formally state a partial data inverse problem concerning $\Lc_{A^{(0)},A^{(1)},A^{(2)}, A^{(3)}}$.
\begin{theorem}\label{Main-thm}
Let $\Omega\subset \R^{n}$, $n\geq2$, be a bounded open set with smooth boundary, and let $\Gamma\subset \partial\Omega$ be an arbitrary closed set. 
Let $\mathcal{L}_{A^{(0)},A^{(1)}, A^{(2)}, A^{(3)}}$ and $\mathcal{L}_{\widetilde{A}^{(0)}\widetilde{A}^{(1)},\widetilde{A}^{(2)},\widetilde{A}^{(3)}}$ be two operators defined as in \eqref{Non_Linear_Operator} with the coefficients satisfying Assumption~\ref{Assumption}. If the corresponding partial DN maps satisfy
    \[ \mathcal{N}^{\Gamma}_{A^{(0)},A^{(1)},A^{(2)},A^{(3)}}(f,g)=\mathcal{N}^{\Gamma}_{\widetilde{A}^{(0)},\widetilde{A}^{(1)},\widetilde{A}^{(2)},\widetilde{A}^{(3)}}(f,g),\]
    for all $(f,g)\in V_\delta(\partial \Omega)$ with $(f,g)|_{\Gamma}=0$, then $A^{(l)}=\widetilde{A}^{(l)}$, for $l=0,\cdots,3$, in $\overline{\Omega}\times \C$.
\end{theorem}

\begin{remark}
A closely related partial data inverse problem is to determine the nonlinear coefficients of \eqref{Non_Linear_Operator} from the Dirichlet data supported in $\Sigma_1$ and the Neumann data measured on $\Sigma_2$, where $\Sigma_1,\Sigma_2$ are arbitrary nonempty open subsets of $\p\O$. As long as $(\Sigma_1 \cap \Sigma_2)$ is nonempty, Theorem \ref{Main-thm} can be applied on $\Gamma = (\Sigma_1 \cap \Sigma_2)^c$ to determine all the coefficients.
\end{remark}

\begin{remark}
Though it is an open question at the moment, the authors believe that the inverse problem in Theorem \ref{Main-thm} can be generalized for a polyharmonic operator $(-\Delta)^m$ for $m\geq 2$ with higher order tensorial perturbations for higher values of $m$.
\end{remark}

Partial data inverse problems, that is inverse problems with information known only on a part of the boundary, are arguably more suitable for the purposes of practical applications. Starting with the phenomenal work \cite{KSU2007}, partial data inverse problems have gained lots of attention in the last few decades. Some of the early progresses came from \cites{ziqi_Sun,DKJU2007 ,KLU_biharmonic2012,Ghosh_Venky2016,Assylbekov_Yang_polyharmonic} and the references therein, where the Neumann data is measured only on a part of the boundary without any restriction on the Dirichlet data. Significant progress has been made in inverse problems with restricting the Dirichlet as well as the Neumann data (see, for instance, \cites{KSU2007,Isakov2007,DSFKSU,Chung_partial2014,Bhattacharyya_partial2018,Mihajlo2020,Katya-Uhlmann-gradient-nonlinearities,Katya-Uhlmann-semilinear-elliptic,Bhattacharyya-Kumar_local_data_polyharmonic}). On the other hand, \cites{KLU_biharmonic2012,KLU_polyharmonic2014} pioneered the study of inverse problems recovering perturbations of polyharmonic operators and were soon followed by an extensive literature, \cites{Ghosh_Venky2016,Assylbekov_Yang_polyharmonic,Bhattacharyya-Ghosh_polyharmonic,Bhattacharyya-Ghosh_biharmonic,Bhattacharyya-Krishnan-Sahoo_polyharmonic,Bhattacharyya-Kumar_local_data_polyharmonic,BKSU2023} to name a few. It is worth mentioning that the recent work \cite{BKSU2023} recovers all the nonlinear perturbations of a biharmonic operator governed by third and lower order symmetric tensors from the full boundary information in three or higher dimensions.
A number of recent articles show that nonlinearity might help when solving inverse problems for Partial Differential Equation(PDE). Starting with \cite{KLU18}, similar phenomena have been observed and utilized in \cite{Feizmohammadi_Oksanen_2020,LLLS21}. We refer the reader to \cites{Kian-katya-Uhlmann_quasilinear_conductivity,Katya-Uhlmann-gradient-nonlinearities,Katya-Uhlmann-semilinear-elliptic,KU_CTA23,KUY_complex24,lai2020partial,MU20,LLLS21,LLST_Fractional_22,LL23,Salo_Tzou23} for the recent progresses on inverse problems for nonlinear elliptic PDE.

\subsection*{Novelty of the article}Here we take the opportunity to discuss the contributions of this article. Firstly, we prove density results in Theorem \ref{MR1} and Theorem \ref{MR2} involving second and third order symmetric tensors which, in contrast to many of the existing results, do not assume any boundary conditions for the tensors. Although we do not pursue optimal regularity of the tensor fields $A^{(l)}$, it can be substantially reduced to $C^{4}(\overline{\Omega})$.
To the best of the authors knowledge, Theorem \ref{MR1} and Theorem \ref{MR2} are the first density results for a general symmetric second and a third order tensor fields using biharmonic functions vanishing on an arbitrary closed subset of the boundary.
Finally, in Theorem \ref{Main-thm} we discuss an inverse problem for a biharmonic operator with lower order nonlinear perturbations up to order three. We describe how the density results are crucially used to solve the inverse problem. To the best of the authors knowledge, Theorem \ref{Main-thm} is the first result for biharmonic operators with nonlinear perturbations from partial boundary data.
It is worth mentioning here that all the results in this article hold true in any dimensions greater or equals to two. This gives a unified method for inverse problems of recovering nonlinear perturbations in any dimension.

\subsection*{Outline of the article}The rest of the article is organized as follows. We collect and review some preliminary results in Section~\ref{prelim}, which are crucial in our proofs. Section~\ref{proofs} is dedicated to the proofs of Theorems~\ref{MR1} and \ref{MR2}. More precisely, subsection~\ref{Sec_reduction} shows that Theorem~\ref{MR1} follows from Theorem~\ref{MR2}. Subsections~\ref{Sec_local} and \ref{Sec_UCP} contain proofs of local and global versions of Theorem~\ref{MR2} respectively. Finally in Section~\ref{IP}, these density results are utilized to prove Theorem~\ref{Main-thm}.

\section{Preliminaries}\label{prelim}
In this section, we present some results for easy reference and convenience of the reader. Some of these results are new and others taken from previous literature, with reference to the source.

\subsection{Well-posedness of the PDE (\ref{Dirichlet_Problem})}\label{well-posed}
For a result on the well-posedness of \eqref{Dirichlet_Problem}, in dimensions $3$ or higher, see \cite{BKSU2023}. The arguments present there seamlessly extend for well-posedness in dimension $2$ as well, see \cite{Katya-Uhlmann-gradient-nonlinearities,LLLS21}. For the sake of completeness, here we state the well-posedness result and provide a brief sketch of the proof.
\begin{theorem}\cite[Theorem A.1]{BKSU2023}\label{th:well_posedness}
    Let $\Omega\subset \mathbb{R}^n$, $n\geq2$, $A^{(l)}:\overline{\Omega}\times \C\to S^{l}$, $l=0,1,2,3$ satisfying Assumption~\ref{Assumption}. Let us denote $C^{k,\alpha}(\overline{\Omega})$ and $C^{k,\alpha}(\partial\Omega)$, for $k\in \mathbb{N}\cup \{0\}$, $0<\alpha<1$, to be the standard H\"older spaces of functions on $\overline{\Omega}$ and $\partial\Omega$, respectively. Then the following statements are true
\begin{itemize}
    \item [1.] There exists $\delta>0$, $C>0$ such that for any $(f,g) \in V_{\delta}(\p\O)$ (see \eqref{eq_v_delta}),
    the boundary value problem \eqref{Dirichlet_Problem} has a solution  $u=u_{f,g}\in C^{4,\A}(\overline{\O})$ satisfying
\[  \lVert u \rVert_{C^{4,\A}(\overline{\O})} 
\leq C\left(\lVert f\rVert_{C^{4,\alpha}(\p\O)} 
+\lVert g\rVert_{C^{3,\alpha}(\p\O)} \right).
\]
\item [2.] The solution $u_{f,g}$ is unique in the class 
\[  \{u\in C^{4,\alpha}(\overline{\O}): \lVert u \rVert_{C^{4,\alpha}(\overline{\O})} <C\delta\},
\] and it depends holomorphically on  $(f,g)\in V_{\delta}(\p\O)$.
\item [3.] The map $V_{\delta}(\p\O)\ni (f,g)\longmapsto(\p_{\nu}^2u|_{\p\O},\p_{\nu}^3 u|_{\p\O}) \in C^{2,\alpha}(\p\O) \times  C^{1,\alpha}(\p\O) $ is holomorphic.
\end{itemize}
\end{theorem}
\begin{proof}[Sketch of proof]
    Due to Assumption~\ref{Assumption}, the coefficients can be expanded into the power series
    \[
    A^{(l)}(\cdot, z) = \sum\limits_{k=1}^\infty \frac{1}{k!} z^k \PD_z^k A^{(l)} (\cdot, 0),
    \]
    converging in $C^{m,\alpha}(\overline{\O}; S^{(l)})$ topology for each $m$. Consider the map 
    \begin{equation}
        \begin{split}
            F: C^{4,\alpha}(\PD \O) \times &C^{3,\alpha}(\PD \O) \times C^4(\overline{\O}) \to C^{0,\alpha}(\overline{\O}) \times C^{4,\alpha}(\PD \O) \times C^{3,\alpha}(\PD \O) \\
            F(f,g,u) &= (\Lc_{A^{(0)}, A^{(1)}, A^{(2)}, A^{(3)}} u, u\lvert_{\PD \O} - f, \PD_\nu u\lvert_{\PD \O} - g).
        \end{split}
    \end{equation}
    As in \cite{BKSU2023}, $F$ is holomorphic as a map between Banach spaces (equipped with the product norm and the induced topology), and its partial differential 
    \begin{equation}
        \PD_u F(0,0,0) : C^{4,\alpha}(\overline{\O}) \to C^{0,\A}(\overline{\O}) \times C^{4,\A}(\PD \O) \times C^{3,\A}(\PD \O)
    \end{equation}
    is given as
    \[
    \PD_u F(0,0,0) v = ((-\Delta)^2 v, v\lvert_{\PD \O}, \PD_\nu v\lvert_{\PD \O}),
    \]
    which is a linear isomorphism (\cite[Theorem~2.19]{Gazzola_Polyharmonic_book_2010}). Invoking Implicit function theorem (see \cite[p. 144]{Inverse_spectral_theory}), we find that there exists $\delta > 0$ and a unique holomorphic map $S: V_\delta(\PD \O) \to C^{4,\A}(\overline{\O})$ such that $S(0,0) = 0$ and $F(f,g, S(f,g)) = 0$ for all $(f,g) \in V_\delta(\PD \O)$. Setting $u = S(f,g)$ and noting that $S$ is Lipschitz continuous, we obtain the desired estimate.
\end{proof}

\subsection{Calculus of symmetric tensor fields}
In this subsection we prove a few tensor algebraic results which play a crucial role in handling the tensorial coefficients (see Lemma \ref{local}).
To the best of the authors knowledge, the results in this subsection (excluding Lemma \ref{decomposition}) are not available in the literature.
We start with recalling standard notations for contraction and its dual for symmetric tensors.
We denote $i_\d: S^{j} \to S^{j+2}$ and its dual $j_\d: S^j \to S^{j-2}$ as
\begin{align*}
    (i_\d f)_{i_1 \dots i_{j+2}} = \sigma \lb f_{i_1 \dots i_{j}} \otimes \d_{i_{j+1} i_{j+2}} \rb, \qquad
    (j_\d f)_{i_1 \dots i_{j-2}} = \sum\limits_{i_{j-1}, i_j = 1}^n f_{i_1 \dots i_{j}} \d_{i_{j-1} i_{j}},
\end{align*}
where $\otimes$ denotes the tensor product of two tensors, $\sigma$ denotes the symmetrization operator and $\d_{ij}$ denotes the Kronecker delta tensor. Any symmetric tensor (or tensor field) satisfying $j_\d f = 0$ is called \emph{trace-free}. By convention, $j_\d f = 0$ for $f \in S^{0}$ and $S^1$. Let us recall the decomposition of a symmetric tensor field into isotropic and trace-free parts (see \cite{DS,PSU}).
\begin{lemma}\label{decomposition}
    A tensor field $A^{(l)} \in C^\infty(\overline{\O}; S^l)$ can be uniquely written as
    \begin{equation}
    A^{(l)}(x) = \sum\limits_{k=0}^{[l/2]} i_\d^k A^{(l)}_{l-2k}(x),
    \end{equation}
    where $A^{(l)}_{l-2k} \in C^\infty(\overline{\O}; S^{l-2k})$, and $j_\d A^{(l)}_{l-2k} = 0$ for each $k \geq 0$.
\end{lemma}
\noindent We will need the above lemma for $l=2,3$, where we denote the decomposition as
\begin{equation}\label{dec-en}
A^{(l)}(x) = \Tilde{A}^{(l)}(x) + i_\d a^{(l-2)}(x).
\end{equation}
\newline
Let us define
\[
\Vc = \{ \xi \in \Cb^n: \xi \cdot \xi = 0, \,\, \text{and} \,\, \mathrm{Im}(\xi_1) \geq 0\}.
\]
We now show that for a symmetric $m$-tensor field $F$, $\langle F, \xi^{\odot m} \rangle= 0$ for $\xi \in \Vc$ is sufficient to claim that $F$ is isotropic where $m =2, 3$.
In other words, the trace-free part of $F$ in the decomposition \eqref{dec-en} vanishes.
For a tensor $F$ of rank $m$, one obviously has 
\[
\langle F, \xi^{\odot m} \rangle= 0 \mbox{ for all } \xi \in \Cb^n \implies F = 0.
\]
However, in our proofs, we do not have access to all the vectors in $\Cb^n$. The following results show that the above implication also holds for trace-free tensors, using only the vectors in $\Vc$ for $m=1,2$ and $3$. Since $\Vc$ is not a linear space, we proceed by explicitly choosing vectors in $\Vc$ and show that each component of $F$ is $0$. Due to this, the calculation grows with $m$. We start with the simplest case of a vector and then move on to higher order symmetric tensors. Let us denote by $e_i$ the $i-$th coordinate vector in $\Rn$.
\begin{lemma}\label{1tensor}
    Let $H$ be a vector such that $\langle H, \xi \rangle = 0$ for all $\xi \in \Vc$. Then $H=0$.
\end{lemma}
\begin{proof}
    By considering $\xi = \I e_1 + e_j$ and $\xi = \I e_1 - e_j$ for $2 \leq j \leq n$, we find
    \begin{align*}
        0 &= \I H_1 + H_j,\\
        \text{and} \quad 0 &= \I H_1 - H_j.
    \end{align*}
    From this, we conclude that $H = 0$.
\end{proof}

\begin{remark}
    Let $H$ be a vector such that for each $1 \leq l \leq n$, $ \xi_l \langle H, \xi \rangle = 0$ for all $\xi \in \Vc$. For any $\xi \in \Vc$, there exists $1 \leq l \leq n$ such that $\xi_l \neq 0$ and hence $\langle H, \xi \rangle = 0$. The preceding lemma then implies that $H=0$.
\end{remark}

\begin{lemma}\label{2tensor}
    Let $G$ be a trace-free symmetric 2-tensor such that $\langle G, \xi^{\odot 2} \rangle = 0$ for all $\xi \in \Vc$. Then $G=0$. 
\end{lemma}
\begin{proof}
    As before, let $j > 1$ and considering $\xi = \I e_1 + e_j$ and then $\xi = \I e_1 - e_j$ gives
    \begin{align*}
        0 &= \lb G_{jj} - G_{11} \rb + 2 \I G_{1j} \\
        \text{and} \quad 0 &= \lb G_{jj} - G_{11} \rb - 2 \I G_{1j}.
    \end{align*}
    These imply that $G_{11} = G_{jj}$ and $G_{1j} = 0$ for each $j > 1$. Since $G$ is trace-free, $G_{jj} = G_{11} = 0$. Thus, we have established that $G_{1k} = 0 = G_{kk}$ for all $k \geq 1$. This finishes the proof for $n=2$.
    Let us assume $n \geq 3$. The only case left to consider is $G_{jk}$, for $1 < j < k \leq n$. Considering $\xi = \I e_1 + \frac{1}{\sqrt{2}} (e_j + e_k)$ and using previous conclusion finishes the proof.
\end{proof}

\begin{lemma}\label{3tensor}
    Let $F$ be a trace-free symmetric 3-tensor such that $\langle F, \xi^{\odot 3} \rangle = 0$ for all $\xi \in \Vc$. Then $F= 0$.
\end{lemma}
\begin{proof}
    \textbf{Step 1:} For $j \neq 1$, Considering $\xi = \I e_1 + e_j \in \Vc$ and $\I e_1 - e_j \in \Vc$, we find
    \begin{align*}
        0 &= \I \lb 3 F_{jj1} - F_{111} \rb + \lb F_{jjj} - 3F_{j11} \rb \\
        \text{and} \quad 0 &= \I \lb 3 F_{jj1} - F_{111} \rb - \lb F_{jjj} - 3F_{j11} \rb.
    \end{align*}
    From this, we conclude that 
    \[
    F_{111} = 3 F_{jj1} \quad \text{and} \quad F_{jjj} = 3 F_{j11}, \quad \text{for each} \quad 2 \leq j \leq n.
    \]
    Since $F$ is trace free, we have
    \begin{align*}
        0 &= \sum\limits_{l=1}^n F_{1ll} 
        = F_{111} + \sum\limits_{l=2}^n F_{1ll}
        = F_{111} + \frac{n-1}{3} F_{111}
        = \frac{n+2}{3} F_{111},
    \end{align*}
    and thus $F_{111} = 0$, which in turn implies $F_{1jj} = 0$, for each $2 \leq j \leq n$. We have established that $F_{1jj} = 0$ for each $1 \leq j \leq n$. 

    If $n=2$, we also have $0 = F_{211} + F_{222} = 4 F_{211}$, and hence $F_{211} = 0 = F_{222}$. Thus, for $n=2$, we have established that $F$ must be $0$.

    \textbf{Step 2:} We now assume that $n \geq 3$. Also in step 1, we have established that $F_{1jj} = 0$ for each $1 \leq j \leq n$. For $2 \leq m < l \leq n$ (such $m,l$ exist since $n \geq 3$), for $a,b \in \R$ chosen suitably (later) such that $\xi = \I e_1 + a e_l + b e_m \in \Vc$, we have (using the result in step 1):
    \begin{align*}
        0 &= \langle F, \xi^{\odot 3} \rangle \\
        &= -3 a F_{l11} - 3b F_{m11} + 6 \I ab F_{ml1} + b^3 F_{mmm} +a^3 F_{lll} + 3 a^2 b F_{mll} + 3 ab^2 F_{mml}. 
    \end{align*}
    Choosing $\lb a,b \rb = \lb \frac{1}{\sqrt{2}}, \frac{1}{\sqrt{2}} \rb$ and then $(a,b) = \lb \frac{-1}{\sqrt{2}}, \frac{-1}{\sqrt{2}} \rb$, and comparing, we have
    \[
    F_{ml1} = 0 \quad \text{for all} \quad 2 \leq m < l \leq n,
    \]
    since this is the only term that doesn't change sign. This then implies that
    \[
    0 = -3 a F_{l11} - 3b F_{m11}  + b^3 F_{mmm} +a^3 F_{lll} + 3 a^2 b F_{mll} + 3 ab^2 F_{mml}.
    \]
    Choosing $(a,b) = \lb \frac{1}{\sqrt{2}}, \frac{1}{\sqrt{2}} \rb$ and then $(a,b) = \lb \frac{1}{\sqrt{2}}, \frac{-1}{\sqrt{2}} \rb$ and comparing, we have
    \[
    - 6 F_{m11} + F_{mmm} + 3 F_{mll} = 0, \quad \text{for all} \quad 2 \leq m< l \leq n.
    \]
    Using $F_{mmm} = 3 F_{11m}$ from step 1, the above equality gives 
    \[
    F_{mll} = F_{m11} = \frac{1}{3} F_{mmm}.
    \]
    Since $F$ is trace free, we have 
    \begin{align*}
        0 &= F_{mmm} + F_{m11} + \sum\limits_{l \neq 1,m} F_{mll} \\
        &= 4 F_{m11} + (n-2) F_{m11},
    \end{align*}
    and thus we have $F_{m11} = F_{mmm} = F_{mll} = 0$. To summarize, we have shown that $F_{ijk} = 0$ when at least two indices are same, or when at least one index is $1$. This finishes the proof if $n=3$.

    \textbf{Step 3:} The only case left to consider is  $F_{ijk}$, when $ n \geq 4$ and when all three indices are distinct and larger than $1$. Considering $\xi = \I e_1 + \frac{1}{\sqrt{3}} (e_i + e_j + e_k)$ for $1 < i < j < k \leq n$ and using previous results finishes the proof.
\end{proof}

\subsection{Some existing density results using biharmonic functions}
Let us recall the following lemma which proves a Runge type density result for biharmonic functions. This lemma serves a crucial purpose in the construction of solutions with desired boundary conditions in the local-to-global proof of Theorem~\ref{MR2} (see Lemma~\ref{Lem_UCP} and its proof). 
\begin{lemma}\cite[Lemma~3.1]{agrawal2023linearized}\label{Runge}
    Let $\O_1 \subset \O_2$ be two bounded open sets with smooth boundaries. Let $G_2$ be the Green kernel for the biharmonic operator associated to the open set $\O_2$, i.e.
    \begin{equation}
    \begin{aligned}  
    (-\Delta_y)^2 G_2(x,y) &= \d(x-y), \quad &&\text{for all~} x,y \in \O_2, \\
    \lb G_2(x, \cdot), \frac{\PD G_2}{\PD \nu}(x, \cdot) \rb &= 0, \quad &&\text{on~} \PD \O_2.
    \end{aligned}
    \end{equation}
    Define the set
    \[
    \Fc \coloneqq \left\{ \int\limits_{\O_2} G_2(\cdot, y) a(y) \, \D y~:~ a \in C^\infty(\overline{\O}_2), \quad \mathrm{supp}(a) \subset \O_2 \setminus \overline{\O}_1 \right\}.
    \]
    Then the set $\Fc$ is dense in the subspace of all biharmonic functions $u \in C^\infty(\overline{\O}_1)$ such that $u\lvert_{\PD \O_1 \cap \PD \O_2} = 0 = \frac{\PD u}{\PD \nu}\lvert_{\PD \O_1 \cap \PD \O_2}$, equipped with $L^2(\O_1)$ topology.
\end{lemma}

Let us also recall the main result of \cite{agrawal2023linearized}. The following result concerning the density of biharmonic functions allows us to infer pointwise information from the integral identity.  Here, $\Ec_\Gamma$ is as defined in \eqref{eq_Ec}.
\begin{theorem}\label{AJS}
    Let $a^2, a^0 \in C^\infty(\overline{\O})$ and $a^1 \in C^\infty(\overline{\O}; \Rn)$. Suppose 
    \[
    \int\limits_{\O} \lb a^2 \Delta u + a^1 \cdot \nabla u + a^0 u \rb v \D x = 0
    \]
    holds for all $u,v \in \Ec_{\Gamma}$, then $a^j = 0$ in $\O$, for $j=0,1,2$.
\end{theorem}

\section{Proofs of the density theorems}\label{proofs}
In this section we prove the density results presented in Theorem \ref{MR1} and Theorem \ref{MR2}. Firstly, in Subsection \ref{Sec_reduction} we prove Theorem \ref{MR1} assuming Theorem \ref{MR2} (see Corollary \ref{Cor_Th_1.1}). Next, in Subsection \ref{Sec_local} we prove Theorem \ref{MR2} locally near the boundary region where the biharmonic functions do not vanish. 
Finally, in Subsection \ref{Sec_UCP}, we prove that under the assumptions of Theorem~\ref{MR2}, we can uniquely continue the tensor fields to vanish over the whole domain $\Omega$.

\subsection{Reduction of Theorem~\ref{MR1} to Theorem~\ref{MR2}.}\label{Sec_reduction}
In this subsection, we show that the density result of Theorem~\ref{MR1} can be reduced to that of Theorem~\ref{MR2}. The main result of this subsection is the following lemma, which gets rid of the intertwining of the derivatives of various biharmonic functions.

Let us define an auxiliary operator $\Ac_m^{\sharp}$ as follows
\[
\Ac_m^{\sharp}(v)(x) \coloneqq \sum\limits_{l=1}^3 \sum\limits_{i_1, \dots, i_l = 1}^{n} A^{(l)}_{i_1 \dots i_l}(x,v) D^{(l)}_{i_1 \dots i_l} v (x) + \frac{1}{m} A^{(0)} (x,v) v(x). 
\]
We have the following lemma regarding the operator $\Ac_m^\sharp$.
\begin{lemma}\label{equiv}
    For $m \geq 2$, the following identities are equivalent.
    \begin{equation}\label{id1}
    \int\limits_{\O} \Ac_{m} (v_0, \dots, v_m) \, \D x = 0, \quad \text{for all } v_i \in \Ec_{\Gamma}, i = 0, 1, \dots , m
    \end{equation}
    \begin{equation}\label{id2}
    \int\limits_{\O} \Ac_m^\sharp (v_0)~v_1 \dots v_m \, \D x = 0,\quad \text{for all } v_i \in \Ec_{\Gamma}, i = 0, 1, \dots , m.
    \end{equation}
    In either case, we have that $\Ac_m^\sharp(v) \equiv 0$ in $\overline{\O}$ for all $v \in \Ec_\Gamma$.
\end{lemma}
\begin{proof}
    Let us assume~\eqref{id1}. It is easily verified that
\[	\begin{aligned}
    \Ac_m^\sharp (v_0) v_1 \dots v_m =& \frac{1}{m} \left[ \Ac_{m}(v_0, \dots, v_m) \right. \\
    &\quad \left.- \left \{ \sum\limits_{k=1}^m \lb \Ac_{m}(v_0, \dots, v_m) - \Ac_{m}(v_k, v_1, \dots, v_0, \dots, v_m) \rb \right\} \right].
\end{aligned}	\]
Hence we arrive at \eqref{id2}.
Conversely, let us assume \eqref{id2}.
Since $m \geq 2$, the density of biharmonic functions with Dirichlet data vanishing on an arbitrary closed set (Theorem~\ref{AJS}) implies
    \[
    \Ac_m^\sharp (v_0) v_1 \dots v_{m-2} \equiv 0, \quad \text{for all } v_i \in \Ec_{\Gamma}, i = 0, \dots, m-2.
    \]
    Since the set of zeros of biharmonic functions (and hence their finite products) has measure zero, we deduce
    \[
    \Ac_m^\sharp (v) = 0 \text{ a.e.}, \quad \text{for all } v \in \Ec_{\Gamma}.
    \]
    Owing to the regularity of $v \in \Ec_{\Gamma}$, the last claim follows.
    By the equality
    \[
        \Ac_{m}(v_0, \dots, v_m) = \sum\limits_{k=1}^m \Ac_m^\sharp(v_k) \lb \prod\limits_{\substack{j=1\\ j \neq k}}^{m} v_j \rb v_0,
    \]
    we obtain \eqref{id1}.
\end{proof}
\begin{corollary}\label{Cor_Th_1.1}
Notice that $\Ac_m^\sharp$ is a special case of the operator $\Ac^{\sharp}$. Thus, Theorem~\ref{MR1} follows from Theorem~\ref{MR2}.
\end{corollary}

\begin{remark}
If we had $\Ac^\sharp (v)$ vanishes identically for \emph{any} smooth function $v$, we then have (by density of smooth functions) that the operator $\Ac^\sharp$ is identically zero, i.e., the coefficients vanish identically. However, we only have that the operator $\Ac^\sharp$ vanishes on $\Ec_{\Gamma}$. Considering $v \in \Ec_{\Gamma}$ in the next section we prove that $\Ac^\sharp$ vanishes locally away from the inaccessible part $\Gamma$. Later, in Subsection \ref{Sec_UCP} we show that the local result can be extended on the whole domain $\O$ using unique continuation arguments.
\end{remark}

We will now focus exclusively on the proof of Theorem~\ref{MR2}.

\subsection{A local version of Theorem \ref{MR2}}\label{Sec_local} 
In this subsection we prove that the coefficients vanish near points on the accessible part $\p\O \setminus \Gamma = \Sigma$ of the boundary. 
Our strategy here is to construct biharmonic functions vanishing, up to order one, on a closed subset of the boundary (see Lemma \ref{construction}). Using these special solutions and using a specific decomposition of the symmetric tensor fields, in Lemma \ref{local} we prove that the tensor fields vanish locally near a part of the boundary.

For simplicity of the arguments, let us consider our domain to be as follows.
Fix $x_0 \in \Sigma$. Without loss of generality (see Remark \ref{difficult_remark}),
we assume that $x_0$ is the origin, $\overline{\O} \subset B(-e_1,1), \Gamma \subset \{x_1 < -c\}$ for some $c>0$ and that the hyperplane $\{x_1 = 0\}$ is the tangent plane to $\O$ at the origin $\in \Sigma$. Hence, $\PD \O \cap \PD (B(-e_1, 1)) = \{0\} \subset \Sigma$. We need to construct biharmonic functions which vanish on $\Gamma$. We now prove that the coefficients vanish in a neighborhood of the origin in this new coordinate system, see Lemma~\ref{local} for the precise statement.
\begin{remark}\label{difficult_remark}
For an arbitrary open bounded domain with smooth boundary, we can use a change of variables to obtain the geometric setup as described above. Since we want to retain $(-\Delta)^2$ as our principal operator, we can use a conformal change of variables. One can obtain that by using a generalized Kelvin transform. We omit the construction here since it is explicitly given in \cite[Section 2.1]{agrawal2023linearized}.
\end{remark}

Let $\chi \in C_c^\infty(\R^n)$ be a cut-off function which is $1$ in a neighborhood of $\Gamma$. We choose $\chi$ such a way that $K := (\mathrm{supp} \, \chi \cap \PD \O) \subset \{x_1 < -c\}$. Let $H_K(y) := \sup \{(x \cdot y) \,:\, x\in K\}$ denote the support function of $K$. The construction of a special solution $u \in \Ec_{\Gamma}$ is based on the following lemma (see also \cites{DSFKSU,agrawal2023linearized,lai2020partial}). Recall that
\[
\Vc = \{ \xi \in \Cb^n: \xi \cdot \xi = 0, \,\, \text{and} \,\, \mathrm{Im}(\xi_1) \geq 0\}.
\]

\begin{lemma}\label{construction}
    Let $\xi \in \Vc$ and $a(x) \in C^{\infty}(\overline{\O})$ satisfying
    \[
    \Delta a \equiv \mathrm{constant}, \quad \text{and} \quad \sum\limits_{i,j} \PD_{x_ix_j} a(x) \xi_i \xi_j = 0, \quad \mbox{in }\O.
    \]
    Then, for any $h > 0$, there exists $u \in \Ec_{\Gamma}$ of the form
    \[
    u(x;h) = e^{-\I x \cdot \xi / h} a(x)+ r(x; h),
    \]
    where the remainder term $r$ satisfies the estimate
    \[
    \|r\|_{H^2(\O)} \leq C \lb 1 + \frac{|\xi|^2}{h^2} + \frac{|\xi|^4}{h^4} \rb^{1/2} e^{\frac{1}{h} H_K (\mathrm{Im}(\xi))},
    \]
    with $C>0$ is independent of $\xi$ and $h$.

    Moreover, $r \in C^\infty(\overline{\O})$, and for $m \in \Nb$ with $m \geq \left[ \frac{n}{2} \right] + 2$, satisfies the estimate
    \[
    \|r\|_{C^3(\O)} \leq C \lb 1 + \frac{|\xi|^2}{h^2}\rb^{\frac{m+2}{2}} e^{-\frac{c}{h} \mathrm{Im}(\xi_1)} e^{\frac{1}{h} |\mathrm{Im}(\xi')|},
    \]
    for some constant $C>0$ independent of $\xi$ and $h$.
\end{lemma}
\begin{proof}
The first part of the lemma is proved in \cite[Lemma~2.1]{agrawal2023linearized}. We briefly recall it below since the later part depends on it. 
Fix $\xi$ and $a(x)$ as in the statement of the lemma. Then the function $e^{-\I x \cdot \xi/ h} a(x)$ is biharmonic in $\O$. For $\chi$ as above, consider $r$ to be the solution to the boundary value problem
    \begin{equation}\label{eq_r}
    \begin{aligned}
        \begin{cases}
            (-\Delta)^2 r &= 0 \quad \text{in} \quad \O, \\
            \lb r, \PD_\nu r \rb &= \lb - (a e^{-\I x \cdot \xi/h}) \chi, -\PD_\nu \lb (a e^{-\I x \cdot \xi/h}) \chi \rb  \rb \quad \text{on} \quad \PD \O.
        \end{cases}
    \end{aligned}
    \end{equation}
    The well-posedness of the above problem gives the estimate
    \[
    \|r\|_{H^2(\O)} \leq C \lb 1+\frac{|\xi|^2}{h^2} + \frac{|\xi|^4}{h^4} \rb^{1/2} e^{\frac{1}{h} H_K (\mathrm{Im}\, \xi)}.
    \]
    Note that since $K \subset \{x_1 < -c\}$, when $\mathrm{Im}(\xi_1) \geq 0$, the estimate on $r$ can be written as
    \begin{align*}
        \|r\|_{H^2(\O)} \leq C \lb 1+\frac{|\xi|^2}{h^2} + \frac{|\xi|^4}{h^4} \rb^{1/2} e^{-\frac{c}{h} \mathrm{Im}\xi_1 } \, e^{\frac{1}{h} |\mathrm{Im}(\xi')|},
    \end{align*}
    where $\xi = (\xi_1, \xi')$. This proves the first part of the lemma.

    For the next part, observe that for any $m \in \Nb$ and $\xi \in \Vc$, elliptic regularity ensures that the function $r$ satisfies
    \begin{equation*}
    \begin{aligned}
        \|r\|_{H^{m+2}(\O)} & \leq C \lb\Big\|\chi a e^{-\I x \cdot \xi/h} \Big \|_{H^{m+3/2}(\PD \O)} + \Big\| \PD_\nu \lb \chi a e^{-\I x \cdot \xi/h} \rb \Big\|_{H^{m+1/2}(\PD \O)} \rb \\
        & \leq C \lb \sum_{i=0}^{m+2}\frac{|\xi|^{2i}}{h^{2i}}  \rb^{\frac{1}{2}} e^{-\frac{c}{h} \mathrm{Im}\xi_1 } \, e^{\frac{1}{h} |\mathrm{Im}(\xi')|}\\
        & \leq C \lb 1+\frac{|\xi|^2}{h^2} \rb^{\frac{m+2}{2}} e^{-\frac{c}{h} \mathrm{Im}\xi_1 } \, e^{\frac{1}{h} |\mathrm{Im}(\xi')|}.
    \end{aligned}
    \end{equation*}
This estimate provides an upper bound for $\|r\|_{H^{2+m}(\O)}$. In particular, by Sobolev embedding theorem \cite[Ch.~5, Theorem~6]{evans}, we find that for $m \geq \left[ \frac{n}{2} \right] + 2$, and for $\xi \in \Vc$, we have
   \[
   \|r\|_{C^3(\O)} \leq C \|r\|_{H^{m+2}(\O)}
    \leq C \lb 1+\frac{|\xi|^2}{h^2}\rb^{\frac{m+2}{2}} e^{-\frac{c}{h} \mathrm{Im}\xi_1 } \, e^{\frac{1}{h} |\mathrm{Im}(\xi')|}.
    \]
    This completes the proof since $C$ here is independent of $\xi$ and $h$.
\end{proof}

The main result of this subsection is the following.
\begin{lemma}\label{local}
    Under the assumptions of theorem~\ref{MR2}, the coefficients vanish in a neighborhood of the origin.
\end{lemma}
\begin{proof}
    Using lemma~\ref{equiv}, we have
    \begin{equation}\label{A_m}
        \Ac^\sharp(v) \equiv 0 \quad \text{in } \O \text{ for all } v \in \Ec_{\Gamma}.
    \end{equation}
    For ease of notation, we often suppress the dependence of the coefficients on the variable $x$. For $\xi \in \Vc$, consider 
    \[
        v = e^{-\frac{\I x \cdot \xi}{h}} + r(x;h)
    \] 
    in \eqref{A_m} to get
    \begin{equation}\label{Am1}
        \sum\limits_{l=1}^3 \lb- \frac{1}{h} \rb^l A^{(l)}_{i_1 \dots i_l} \xi_{i_1} \dots \xi_{i_l} + A^{(0)} 
        = - e^{\frac{\I x \cdot \xi}{h}} \lb \sum\limits_{l=1}^3 A^{(l)}_{i_1 \dots i_l} D^{(l)}_{i_1 \dots i_l} r + A^{(0)} r \rb
    \end{equation}
    Recall that the coefficients $A^{(0)},\cdots,A^{(3)}$ are bounded in $\overline{\O}$. For $x \in (\O \cap \{x_1 > -c\})$, and a multi-index $|\beta| \leq 3$, we have
   \begin{align*}
        |(\PD^\beta r)(x) e^{\frac{\I x \cdot \xi}{h}}| & \leq C e^{-\frac{x_1}{h} \mathrm{Im}(\xi_1)} \, e^{\frac{1}{h} |\mathrm{Im}(\xi')|} \, \|r\|_{C^3(\overline{\O})} \\
        & \leq C \lb 1 + \frac{|\xi|^2}{h^2}\rb^{\frac{m+2}{2}} e^{-\frac{(c+x_1)}{h} \mathrm{Im}(\xi_1)} \, e^{\frac{2}{h} |\mathrm{Im}(\xi')|}.
    \end{align*}
    Therefore, multiplying \eqref{Am1} by $h^3$ and letting $h \to 0$, we find
    \[
    \sum_{i,j,k}A^{(3)}_{ijk}(x) \xi_i \xi_j \xi_k = 0,
    \]
    for all $x \in (\O \cap \{x_1 > -c\})$, and $\xi \in \Vc$. Splitting $A^{(3)} = \Tilde{A}^{(3)} + i_\d a^{(1)}$ using its trace free decomposition (Lemma~\ref{decomposition}), we observe
    \[
    \sum\limits_{i,j,k} \Tilde{A}^{(3)}_{ijk}(x) \xi_i \xi_j \xi_k = 0
    \]
    for all $x \in (\O \cap \{x_1 > -c\})$, and $\xi \in \Vc$. Using Lemma~\ref{3tensor}, we conclude that $\Tilde{A}^{(3)}(x)= 0$ for $x \in \O \cap \{x_1 > -c\}$. Since $\sum_{i,j,k}(i_\delta a^{(1)})_{ijk} \xi_i \xi_j \xi_k = 0$ for $\xi \in \Vc$, for $x \in (\O \cap \{x_1 > -c\})$, \eqref{Am1} gives
    \begin{equation}\label{Am2}
        \frac{1}{h^2} \sum_{i,j}A^{(2)}_{ij} \xi_i \xi_j - \frac{1}{h} \sum_{i}A^{(1)}_i \xi_i + A^{(0)} = - e^{\frac{\I x \cdot \xi}{h}} \lb \sum\limits_{l=1}^3 A^{(l)}_{i_1 \dots i_l} D^{(l)}_{i_1 \dots i_l} r + A^{(0)} ~r \rb.
    \end{equation}
Multiplying by $h^2$ and letting $ h \to 0$, we get for $x \in \O \cap \{x_1 > -c\}$ and $\xi \in \Vc$:
    \[
    \sum_{i,j}A^{(2)}_{ij} \xi_i \xi_j = 0.
    \]
    Again using the trace free decomposition of $A^{(2)}$, we observe
    \[
    \sum_{i,j}\Tilde{A}^{(2)}_{ij} \xi_i \xi_j = 0.
    \]
    Using Lemma~\ref{2tensor}, we conclude that $\Tilde{A}^{(2)}(x) = 0$, for $x \in \O \cap \{x_1 > -c\}$. Thus, for such $x$, \eqref{Am2} becomes
    \begin{equation}\label{Am3}
        - \frac{1}{h} A^{(1)}_i \xi_i + A^{(0)} = -e^{\frac{\I x \cdot \xi}{h}} \lb \sum\limits_{l=1}^3 A^{(l)}_{i_1 \dots i_l} D^{(l)}_{i_1 \dots i_l} r + A^{(0)} ~r \rb.
    \end{equation}
    Multiplying by $h$ and then letting $h \to 0$, we get for $x \in \O \cap \{x_1 > -c\}$ and $\xi \in \Vc$:
    \[
    A^{(1)}_i \xi_i = 0.
    \]
    Using Lemma~\ref{1tensor}, we find $A^{(1)}(x) = 0$ for $x \in \O \cap \{x_1 > -c\}$. For such $x$, \eqref{Am3} gives
    \[
    A^{(0)} = -e^{\frac{\I x \cdot \xi}{h}} \lb \sum\limits_{l=1}^3 A^{(l)}_{i_1 \dots i_l} D^{(l)}_{i_1 \dots i_l} r + A^{(0)} ~r \rb.
    \]
    Letting $h \to 0$, we conclude that $A^{(0)}$ also vanishes for $x \in \O \cap \{x_1 > -c\}$.

    Going back to \eqref{A_m}, we get for $x \in \O \cap \{x_1 > -c\}$, 
    \[
    a^{(1)}_i D_i (-\Delta) v + a^{(0)} (-\Delta) v = 0, \quad \forall v \in \Ec_{\Gamma}.
    \]
    Fix $1 \leq l \leq n$, and consider
    \[
    v = x_l e^{-\I x \cdot \xi /h} + r(x;h) \in \mathcal{E}_\Gamma \quad (\text{due to Lemma~\ref{construction}}),
    \] 
    to get
    \[
    \frac{\I}{h^2} a^{(1)}_i \xi_i \xi_l - \frac{\I}{h} a^{(0)} \xi_l = -\frac{1}{2} e^{\I x \cdot \xi/h} \lb a^{(1)}_i D_i (-\Delta) r + a^{(0)} (-\Delta) r \rb.
    \]
    Proceeding as before, with this we conclude that $a^{(1)}$ and $a^{(0)}$ vanish in $\O \cap \{x_1 > -c\}$. Thus, we have shown that the tensor fields $A^{(0)},\cdots,A^{(3)}$ vanish in a neighborhood of the origin, namely $\{x_1 > -c\}$.
\end{proof}

\subsection{The coefficients vanish everywhere}\label{Sec_UCP} 
So far, in Lemma \ref{local} we proved that under the assumptions of Theorem~\ref{MR2}, the coefficients $A^{(k)}$, for $k=0,1,2,3$, vanish in a neighborhood of some point $x_0 \in \p\O$.
Here we prove Theorem~\ref{MR2} by showing that, under the same assumptions, we can uniquely continue to extend the neighborhood to the entire $\Omega$. The method is based on generalizing \cite[Section 2]{DSFKSU} for biharmonic functions.

\begin{proof}[Proof of Theorem~\ref{MR2}]
We prove that the coefficients $A^{(k)}$, for $k=0,\dots,3$, vanish everywhere in $\O$. Fix a point $x_1 \in \O$ and let $\Theta:[0,1] \to \overline{\O}$ be a smooth curve joining $x_0 \in \PD \O \setminus \Gamma$ to $x_1$, and satisfying $\Theta(0) = x_0, \Theta'(0)$ is the interior normal to $\PD \O$ at $x_0$ and $\Theta(t) \in \O$ for all $t \in (0,1]$. For $\epsilon > 0$, consider the closed neighborhood of the curve ending at $\Theta(t)$ as
\[
\Theta_\epsilon(t) \coloneqq \{x \in \Rb^n~:~ \mathrm{dist}(x, \Theta ([0,t])) \leq \epsilon\}. 
\]
Define the set 
\[
I_{\epsilon} = \Big\{t \in [0,1]~:~ \lb A^{(3)}, A^{(2)}, A^{(1)}, A^{(0)} \rb ~\text{vanish a.e. in}~ \Theta_\epsilon(t) \cap \O\Big\}.
\]
By the local uniqueness Lemma~\ref{local}, for $\epsilon > 0$ small enough, $I_{\epsilon}$ is non-empty. It is easy to see that $I_{\epsilon}$ is a closed subset of $[0,1]$. In Lemma~\ref{Lem_UCP} below, we prove that $I_{\epsilon}$ is also an open subset of $[0,1]$ for $\epsilon>0$ small enough.

So, assuming Lemma~\ref{Lem_UCP} for the moment, we see that $I_{\epsilon} \subset [0,1]$ is open as well as closed. Therefore, $I_{\epsilon} = [0,1]$ for some $\epsilon>0$ and hence the coefficients $A^{(k)}$ vanishes in an open neighborhood of $x_1$. Since we choose $x_1 \in \O$ arbitrarily, the coefficients must vanish on entire $\Omega$.
\end{proof}

To complete the proof of Theorem~\ref{MR2}, now in Lemma~\ref{Lem_UCP} we show that $I_{\epsilon}$ is open.

\begin{lemma}\label{Lem_UCP}
Under the assumptions of Theorem~\ref{MR2} and the notations defined above, $I_{\epsilon}$ is an open subset of $[0,1]$ for small enough $\epsilon>0$.
\end{lemma}
\begin{proof}
Without loss of generality, let us assume that $\epsilon > 0$ is small enough such that $\Theta_\epsilon(1) \subset \overline{\O}$.
If $t \in I$, there is $\epsilon>0$ small enough such that $\left(\PD \Theta_\epsilon(t) \cap \PD \O\right) \subset (\PD \O \setminus \Gamma)$. Let $\O_1 \subset \O$ be an open neighborhood of $\left(\O \setminus \overline{\Theta_\epsilon(t)}\right)$, with smooth and connected boundary, in $\O$. One can check that $\Gamma \subset (\PD \O \cap \PD \O_1)$.
Let us define $\O_2 \supset (\O \cup B(x_0,\epsilon'))$ with $\epsilon'>0$ small enough, such that $\O_2$ is open, connected, $\PD\O_2$ is smooth and
\[
(\PD \O_2 \cap \PD \O) \supset (\PD \O_1 \cap \PD \O) \supset \Gamma.
\]
For a nice visual description of the sets $\O_1$ and $\O_2$, see \cite[Figure 1, Lemma 2.2]{DSFKSU}, see also \cite{Katya-Uhlmann-gradient-nonlinearities}.

Let $G_2$ denote the Green kernel associated to the open set $\O_2$ corresponding to the bilaplacean, i.e., for each $x \in \O_2$,
\[\begin{aligned}
    (-\Delta_y)^2 G_2 (x,y) &= \delta(x-y), \quad &&\mbox{in }\O_2, \\
    (G_2(x,\cdot),\PD G_2(x,\cdot))  &= (0,0), \quad &&\mbox{on }\PD\O_2.\\
    \end{aligned}\]
    For simplicity, let us denote $U(y,z_2) = \Ac^\sharp (y,D_y) G_2(y,z_2)$ for $y \in \O_1$ and $z_2 \in \O_2 \setminus \overline{\O_1}$. For $z_0, z_1, z_2 \in \O_2 \setminus \overline{\O_1}$, the function
    \[
    (z_0, z_1,z_2) \mapsto \int\limits_{\O_1} U(y,z_2) G_2(y,z_1) G_2(y,z_0) ~\D y
    \]
    is biharmonic in each $z_r$, $r=0,1,2$. Furthermore, 
    \[
    \int\limits_{\O_1} U(y,z_2) G_2(y,z_1) G_2(y,z_0) ~\D y = \int\limits_{\O} U(y,z_2) G_2(y,z_1) G_2(y,z_0) ~\D y,
    \]
    since the coefficients in $U$ vanish on $\O \setminus \O_1 \subset \Theta_\epsilon(t)$. For $z_r \in \O_2 \setminus \overline{\O}$, $r=0,1,2$, the functions $y \mapsto G_2(z_r,y)$ are smooth in $\overline{\O}$, biharmonic in $\O$, and vanish, along with their normal derivatives, on $\Gamma$ (i.e. they are in $\Ec_{\Gamma}$). Thus, by the assumption of Theorem~\ref{MR2}, we obtain
    \begin{equation}\label{UCP_eq_1}
    \int\limits_{\O} U(y,z_2) G_2(y,z_1) G_2(y,z_0) ~\D y = 0, \quad \mbox{for all } z_0,z_1,z_2  \in \O_2 \setminus \overline{\O}.
    \end{equation}
Multiplying by $a(z_0) \in C^\infty(\overline{\O_2})$ such that $\mathrm{supp} (a) \subset \O_2 \setminus \overline{\O_1}$, and integrating \eqref{UCP_eq_1} over $z_0 \in \O_2$, we obtain
    \[
    \int\limits_{\O_2} \int\limits_{\O_1} U(y,z_2) G_2(y,z_1) G_2(y,z_0) a(z_0)~\D y ~\D z_0 = 0.
    \]
    Using Fubini's theorem and invoking Lemma~\ref{Runge}, we find
    \[
    \int\limits_{\O_1} U(y,z_2) G_2(y,z_1) u_0(y) ~\D y = 0,
    \]
    for any $u_0 \in \Ec_1$, where
    \[
    \Ec_1 \coloneqq \{u \in C^\infty(\O_1)~:~ (-\Delta)^2 u = 0, \lb u, \PD_\nu u \rb\lvert_{\PD \O_1 \cap \PD \O_2} = 0\}.
    \]
    Similarly, multiplying again by a smooth function of $z_1$ supported in $\O_2 \setminus \overline{\O_1}$, and invoking Lemma~\ref{Runge} we get
    \begin{equation}\label{IBP}
    \begin{aligned}
    \int\limits_{\O_1} u_0(y) u_1(y) \Big ( \sum_{k=0}^{3} \sum_{i_1,\dots,i_k=1}^{n} A^{(k)}_{i_1 \dots i_3}(y) D^{(k)}_{i_1 \dots i_3} G_2(z_2,y) \Big) &~\D y\\
    =\int\limits_{\O_1} U(y,z_2) u_1(y) u_0(y) ~\D y &= 0,  \quad \forall u_1, u_0 \in \Ec_1.    \end{aligned}\end{equation}
Note that, here we cannot use the previous method to replace $G_2$ by some $u_2 \in \Ec_1$, since Lemma~\ref{Runge} holds only in the $L^2(\O_1)$ topology. To handle this, we integrate by parts each term in \eqref{IBP} as long as $G_2$ is free of derivatives. We remark that in doing so, no boundary terms appear, which can be seen as follows. Note that $\PD \O_1 = \lb \PD \O_1 \cap \PD \O \rb \cup \lb \PD \O_1 \cap \O \rb$. The integral on $\lb \PD \O_1 \cap \PD \O \rb$ is zero, since any term having at least one of $u_0, u_1, G_2$ along with their normal derivatives vanish on $\lb \PD \O_1 \cap \PD \O_2 \rb \supset \lb \PD \O_1 \cap \PD \O \rb$. The integral on the second part $\lb \PD \O_1 \cap \PD \O \rb$ is zero since the coefficients vanish on $\Theta_\epsilon(t) \supset \lb \O_1^c \cap \O \rb$, and since the coefficients are smooth, the coefficients and all their derivatives vanish on $\PD \O_1 \cap \O = \PD (\O_1^c) \cap \O$. 
Once the $G_2$ term is free of derivatives, we multiply a function from $C^{\infty}_c(\O_2\setminus\overline{\O_1})$ to use Lemma~\ref{Runge} and shift back the derivatives using integration by parts to get
    \[
    0 = \int\limits_{\O_1} u_0(y) u_1(y) \Big ( \sum_{k=0}^{3} \sum_{i_1,\dots,i_k=1}^{n} A^{(k)}_{i_1 \dots i_3}(y) D^{(k)}_{i_1 \dots i_3}u_2(y) \Big) ~\D y, \quad \forall u_0, u_1, u_2 \in \Ec_1.
    \]
    
    By Lemma~\ref{local} now applied to $\O_1$, the above identity now implies that $A^{(k)}$, $k=0,\dots,3$, vanish on a neighborhood of each point of $\PD \O_1 \setminus \lb \PD \O_1 \cap \PD \O_2 \rb$. This implies that the coefficients vanish on a slightly bigger set $\Theta_\epsilon(\tau)$, for some $\tau > t$. Therefore, we have shown that $I_{\epsilon}$ is open.
\end{proof}
\begin{remark}
    In comparison to the higher order density result in \cite{Katya-Uhlmann-gradient-nonlinearities}, the above proof uses only $L^2(\O_1)$ density result as in Lemma \ref{Runge} and uses integration by parts techniques to avoid derivatives over the Green's function $G_2$. Perhaps one can prove a $H^3(\O_1)$ density result, which can provide an alternative proof of Theorem~\ref{MR2}. Nevertheless, the authors believe that the techniques developed above are more direct and can be easily generalized for higher order in appropriate setting.
\end{remark}

\section{Application to inverse problems}\label{IP}
In this section, we prove Theorem ~\ref{Main-thm}.
The idea of the proof broadly follows techniques developed in \cite{BKSU2023} combined with Theorems \ref{MR1} and \ref{MR2}. We first linearize the semilinear problem \eqref{Dirichlet_Problem} about $z=0$ and obtain linear PDEs involving $z$-derivatives of the coefficients $A^{(l)}$. The next step is to use the linearized DN-map to derive integral identities, as in Theorem \ref{MR1}, involving the differences of the derivatives of the unknown coefficients at $z=0$. Finally, using analyticity of the coefficients in the $z$ variable and the density result Theorem \ref{MR1} we complete the proof of Theorem \ref{Main-thm}.
Let us recall from the statement of Theorem \ref{Main-thm} that, we have two Dirichlet problems
\begin{equation}\label{Forward_Prob}
\begin{aligned}
\begin{cases}
\mathcal{L}_{A^{(0)},A^{(1)},A^{(2)},A^{(3)}}u= 0,\quad \mbox{in }\O,\\
(u,\p_{\nu}u) = (f,g), \quad \mbox{on }\p\O.
\end{cases}
\mbox{ and}\quad 
\begin{cases}
\mathcal{L}_{\wt{A}^{(0)}\wt{A}^{(1)},\wt{A}^{(2)},\wt{A}^{(3)}}\wt{u}= 0,\quad \mbox{in }\O,\\
(\wt{u},\p_{\nu}\wt{u}) = (f,g), \quad \mbox{on }\p\O.
\end{cases}
\end{aligned}
\end{equation}
Furthermore, we have assumed $A^{(l)},\wt{A}^{(l)}$ satisfy Assumption~\ref{Assumption}. That is
\begin{equation}\label{holomorphic}
\begin{gathered}
A^{(l)}(\cdot,z) = \sum_{k=1}^{\infty} \frac{1}{k !}z^{k} \p_z^{k} A^{(l)}(\cdot,0); \quad 
\wt{A}^{(l)}(\cdot,z) = \sum_{k=1}^{\infty} \frac{1}{k !}z^{k} \p_z^{k} \wt{A}^{(l)}(\cdot,0),
\quad l=0,1,2,3;
\end{gathered}
\end{equation}
converges uniformly and absolutely in $\O$, for all $z \in \C$. Recall that, for $f,g \in V_{\delta}(\p\O)$ there exist unique solutions $u,\wt{u} \in C^{4,\A}(\overline{\O})$ of \eqref{Forward_Prob}.
We write $A^{(l),k}(\cdot)$ and $\wt{A}^{(l),k}(\cdot)$ to denote $\p_z^kA^{(l)}(\cdot,z)|_{z=0}$ and $\p_z^k\wt{A}^{(l)}(\cdot,z)|_{z=0}$ respectively, for $k=0,1,\cdots$ and $l=0,1,2,3$.
From the statement of Theorem \ref{Main-thm} we further obtain that
\begin{equation}\label{DN_map}
	(u-\wt{u})|_{\p\O} = 0 = \p_{\nu} (u-\wt{u})|_{\p\O} \quad 
	\implies \,
	\p^2_{\nu} u|_{\Gamma^c} = \p^2_{\nu} \wt{u}|_{\Gamma^c} \mbox{ and } \p^3_{\nu} u|_{\Gamma^c} = \p^3_{\nu} \wt{u}|_{\Gamma^c}.
\end{equation}

\subsection{Multi-linearization}
We start with a first order and a second order linearization of \eqref{Forward_Prob} to obtain linear PDEs with coefficients $\p_z^1A^{(l)}(\cdot,0)$ and $\p_z^1\widetilde{A}^{(l)}(\cdot,0)$ along with linearized DN maps.
We solve the corresponding inverse problem for the linearized equations to obtain \begin{equation}\label{Hyp_3}
\p_z^1A^{(l)}(\cdot,0)=\p_z^1\widetilde{A}^{(l)}(\cdot,0) \quad \mbox{for } l=0,\cdots,3 \quad \mbox{in } \overline{\O}.
\end{equation}
Then we perform a third order linearization of \eqref{Forward_Prob} and use \eqref{Hyp_3} to obtain linear PDEs wih coefficients $\p_z^2A^{(l)}(\cdot,0)$ and $\p_z^2\widetilde{A}^{(l)}(\cdot,0)$ for $l=0,\cdots,3$ which helps us to conclude that $\p_z^2A^{(l)}(\cdot,0)=\p_z^2\widetilde{A}^{(l)}(\cdot,0)$ for $l=0,\cdots,3$ in $\overline{\O}$.

So, inductively, given 
\begin{equation}\label{Hyp_m-2}
\p_z^{k}A^{(l)}(\cdot,0)=\p_z^{k}\widetilde{A}^{(l)}(\cdot,0) \quad \mbox{for } k=0,1,\cdots,m-2 \mbox{ and } l=0,\cdots,3, \quad \mbox{in } \overline{\O},
\end{equation}
we linearize \eqref{Forward_Prob} of $m$-th order and show $\p_z^{m-1}A^{(l)}(\cdot,0)=\p_z^{m-1}\widetilde{A}^{(l)}(\cdot,0)$ for $l=0,\cdots,3$ in $\overline{\O}$. Note that, we can start the induction from $m=2$ because we know the hypothesis \eqref{Hyp_m-2} when $m-2=0$.
(This indicates why we must take $m\geq 2$.)

Let us fix $m\geq 2$ and consider $\varepsilon := (\varepsilon_{1}, \cdots, \varepsilon_{m})\in \R^{m}$. Consider the Dirichlet problem \eqref{Dirichlet_Problem} with
\[
f=\sum_{k=1}^{m} \varepsilon_{k}f_{k}, \quad 
g=\sum_{k=1}^{m} \varepsilon_{k}g_{k},
\quad \mbox{ where } f_k,g_k\in C^{\infty}(\PD\O),
\quad f_k|_{\Gamma}, g_k|_{\Gamma} = 0,
\]
for $k=1,2,\cdots,m$.
Hence, for $|\ve|$ sufficiently small, $f,g \in V_{\delta}(\p\O)$.
Let $u:=u(x,\varepsilon)$ be the unique solution (see Theorem \ref{th:well_posedness}) to the following Dirichlet problem
\begin{equation}\label{Eq:1}
\begin{aligned}
(-\Delta)^2u+\sum_{l=0}^{3} \sum_{i_{1},\cdots,i_{l}=1}^n \big( \sum_{k=1}^\infty \partial_z^kA^{(l)}_{i_{1},\cdots, i_{l}}(x;0)\frac{u^k}{k!} \big) D^{(l)}_{i_{1}\dots i_{l}}u
= 0, \quad &\text{in } \Omega,\\
\left(u, \PD_{\nu} u\right)=\left( \sum_{k=1}^{m}\ve_k f_k, \sum_{k=1}^{m}\ve_k g_k \right),\quad &\text {on } \partial\Omega.
\end{aligned}
\end{equation}
Differentiating \eqref{Eq:1} with respect to $\varepsilon_{k}$, $k= 1, \cdots,m$, taking $\varepsilon=0$, and using that $u (x, 0) = 0$, we get
\begin{equation}\label{Eq:2}
\begin{aligned}
\begin{cases}
(-\Delta)^2v_k=0 &  \text{ in } \Omega,\\
\left(v_k, \PD_{\nu} v_k\right) =\left(f_{k}, g_{k}\right) & \text { on } \partial\Omega,
\end{cases}
\end{aligned}
\end{equation}
where $v_k:=\partial_{\varepsilon_k} u|_{\varepsilon=0}$.
Treating the operator $\Lc_{\wt{A}^{(0)},\wt{A}^{(1)},\wt{A}^{(2)}, \wt{A}^{(3)}}$ similarly, for the same Dirichlet data $f,g \in V_{\delta}(\p\O)$ we obtain unique solution $\wt{u} \in C^{4,\A}(\overline{\O})$ of the Dirichlet problem 
\begin{equation}\label{Eq:1-1}
\Lc_{\wt{A}^{(0)},\wt{A}^{(1)},\wt{A}^{(2)}, \wt{A}^{(3)}}\wt{u}=0 \quad \mbox{in }\O \quad \mbox{and}\quad  (\wt{u},\p_{\nu}\wt{u})|_{\p\O} = (f,g).
\end{equation}
Upon differentiating the above equation with respect to $\ve_k$ and taking $\ve=0$ we see that $\wt{v_k}:=\partial_{\varepsilon_k} \wt{u}|_{\ve=0}$ solves \eqref{Eq:2}. By the uniqueness and the regularity of elliptic PDEs, it follows that $v_k = \wt{v_k} \in C^{\infty}(\overline{\Omega})$, for $k=1,\cdots,m$.

Observe that, the first order linearization \eqref{Eq:2} of the Dirichlet problem \eqref{Dirichlet_Problem} does not retain any information about the lower order perturbations, which is expected due to Assumption \ref{Asmp_2}.

Therefore, we move on to higher order linearizations of \eqref{Eq:1}. 
Assume \eqref{Hyp_m-2} for $m\geq 2$ and apply $\p^m_{\ve_1\cdots \ve_m}|_{\ve=0}$ to \eqref{Eq:1}. We write $w_m=\p^m_{\ve_1\cdots\ve_m}u|_{\ve=0}$
to obtain the $m$-th order linearized equation
\begin{equation}\label{Eq:m}
\begin{aligned}
(-\Delta)^2w_m + \sum_{l=0}^3 \sum_{i_1,\dots, i_l=1}^n  
A^{(l),m-1}_{i_1\dots i_l}
\left(\sum_{j=1}^m D^{(l)}_{i_1\cdots i_l}v_j \prod_{\substack{r=1\\r\neq j}}^m v_r \right) + H_m=0 &  \text{ in } \Omega,\\
\left(w_m,\p_{\nu} w_m\right)=(0,0) & \text{ on } \p \Omega,
\end{aligned}
\end{equation}
where $H_m$ consists of terms $A^{(l),k}$ for $k=0,\cdots,m-2$ given as
\[\begin{aligned}
H_m=\p_{\varepsilon_1}\dots\p_{\varepsilon_m}
\bigg(\sum_{l=0}^3 \sum_{i_1,\dots, i_l=1}^n \big( \sum_{k=1}^{m-2} A^{(l), k}_{i_1\dots i_l}(x)\frac{u^k}{k!} \big) D^{(l)}_{i_1\dots i_l}u
\bigg)\bigg|_{\varepsilon=0}.
\end{aligned}
\]
Dealing with \eqref{Eq:1-1} similarly, we obtain
\begin{equation}\label{Eq:m-1}
\begin{aligned}
(-\Delta)^2\widetilde{w}_m + \sum_{l=0}^3 \sum_{i_1,\dots, i_l=1}^n  
\widetilde{A}^{(l),m-1}_{i_1\dots i_l}
\left(\sum_{j=1}^m D^{(l)}_{i_1\cdots i_l}v_j \prod_{\substack{r=1\\r\neq j}}^m v_r \right) + \widetilde{H}_m=0 &  \text{ in } \Omega,\\
\left(\widetilde{w}_m,\p_{\nu} \widetilde{w}_m\right)=(0,0) & \text{ on } \p \Omega,\\
\mbox{with } \quad 
\widetilde{H}_m=\p_{\varepsilon_1}\dots\p_{\varepsilon_m}
\bigg(\sum_{l=0}^3 \sum_{i_1,\dots, i_l=1}^n \big( \sum_{k=1}^{m-2} \widetilde{A}^{(l), k}_{i_1\dots i_l}(x)\frac{\widetilde{u}^k}{k!} \big) D^{(l)}_{i_1\dots i_l}\widetilde{u}
\bigg)\bigg|_{\varepsilon=0}
\end{aligned}
\end{equation}
where $\widetilde{w}_m=\p^m_{\ve_1\cdots\ve_m}\widetilde{u}|_{\ve=0}$ and $\widetilde{A}^{(l),k}(\cdot) = \partial_{z}^k\widetilde{A}^{(l)}(\cdot;0)$ in $\O$.
Note that, for $m=2$, due to Assumption
\ref{Asmp_2} we have $H_m=0=\widetilde{H}_m$. For $m>2$, we
use the induction hypothesis \eqref{Hyp_m-2} and the well-posedness of the linear biharmonic operators to obtain $\partial^k_{{\varepsilon_{j_1}}\cdots{\varepsilon_{j_k}}} (u-\widetilde{u})|_{\varepsilon = 0} = 0$ for $2\leq k < m$ and $H_m=\widetilde{H}_m$ in $\Omega$ (see also \cite{BKSU2023}).
Therefore, taking the difference of \eqref{Eq:m} and \eqref{Eq:m-1} we obtain
\begin{equation}\label{Eq:4}
\begin{aligned}
(-\Delta)^2(\wt{w}_m-w_m)
=\sum_{l=0}^3 \sum_{i_1,\dots, i_l=1}^n W^{(l), m-1}_{i_1\dots i_l} \, \sum_{j=1}^m D^{(l)}_{i_1\cdots i_l}v_j \prod_{\substack{r=1\\r\neq j}}^m v_r, \quad \mbox{in } \O,
\end{aligned}
\end{equation}
where $W^{(l), m-1}_{i_1\dots i_l}:=\left(A^{(l), m-1}_{i_1\dots i_l}-\wt{A}^{(l),m-1}_{i_1\dots i_l}\right)$ for $l=0,1,2,3$ for $m\geq 1$.
Note that, \eqref{DN_map} implies $(w_m-\wt{w}_m)|_{\p\O} = 0 = \p_{\nu}(w_m-\wt{w}_m)|_{\p\O}$ and $\p^2_{\nu}(w_m-\wt{w}_m)|_{\Gamma^c}=\p^3_{\nu}(w_m-\wt{w}_m)|_{\Gamma^c}$.
Recall the set $\Ec_{\Gamma}$ as in \eqref{eq_Ec} and observe that the due to assumptions of Theorem \ref{Main-thm}, $v_1,\cdots,v_m \in \mathcal{E}_{\Gamma}$.
Multiplying \eqref{Eq:4} by $v_0\in \Ec_{\Gamma}$ and performing integration over $\Omega$ we get
\begin{equation}\label{last}
\begin{aligned}
\sum_{l=0}^3 \sum_{i_1,\dots, i_l=1}^n \int_{\Omega} W^{(l), m-1}_{i_1\dots i_l}(x) \, \left(\sum_{j=1}^m D^{(l)}_{i_1\cdots i_l}v_j(x) \prod_{\substack{r=1\\r\neq j}}^m v_r(x)\right) v_0(x)\, dx=&0.
\end{aligned}
\end{equation}
Following the notation in \eqref{Op_Int} let us define
\[	\Wc_{m}^{m-1}(v_0,v_1,\cdots,v_m) := \sum_{l=0}^3 \sum_{i_1,\dots, i_l=1}^n W^{(l), m-1}_{i_1\dots i_l} \, \left(\sum_{j=1}^m D^{(l)}_{i_1\cdots i_l}v_j \prod_{\substack{r=1\\r\neq j}}^m v_r\right) v_0.
\]
Now, \eqref{last} implies
\[	\int_{\O} \Wc_{m}^{m-1}(v_0,v_1,\cdots,v_m) \, dx = 0, \quad \forall v_i \in \Ec_{\Gamma}, \quad i=0,\cdots,m.
\]
Therefore, applying Theorem~\ref{MR1}, we conclude that
\[
A^{(l),m-1}=\wt{A}^{(l),m-1}, \quad l=0,1,2,3, \quad
\text { in } \Omega,
\]
which completes the mathematical induction.
Thus, we have
\[
\partial_z^{m-1} A^{(l)}(x;0)= A^{(l),m-1}=\wt{A}^{(l),m-1} = \partial_z^{m-1}\wt{A}^{(l)}(x;0), \quad
\text { in } \Omega,
\]
for each $l=0,1,2,3$ and for all $m=1,2,\cdots$.

Finally, using analyticity of the coefficients $A^{(l)}(\cdot;z)$ and $\widetilde{A}^{(l)}(\cdot;z)$ in the $z$ variable as in \eqref{holomorphic}, we conclude the proof of Theorem \ref{Main-thm}.
\qed

\section*{Acknowledgment}
S.B. was partially supported by SERB grant SRG/2022/001298. This work was initiated when DA visited IISER-Bhopal during the spring of 2023. DA thanks the Department of Mathematics for the hospitality during the stay. PK was supported by the Senior Research Fellowship from IISER Bhopal.


\def\dbar{\leavevmode\hbox to 0pt{\hskip.2ex \accent"16\hss}d}

\end{document}